%% file: main_plat_den.tex
\documentclass{amsart}

\usepackage{amssymb,
amsmath,
amsthm,
graphicx,
amscd,
multind,
eufrak,
tikz,
url,
hyperref
}
\usetikzlibrary{matrix,arrows,calc}
\usepackage{verbatim}

\theoremstyle{plain}
\newtheorem{thm}{Theorem}[section]
\newtheorem{cor}[thm]{Corollary}
\newtheorem{prop}[thm]{Proposition}

\newtheorem{lem}[thm]{Lemma}
\theoremstyle{definition}
\newtheorem{no}[thm]{Notation}
\newtheorem{ex}[thm]{Example}
\newtheorem{re}[thm]{Remark}
\newtheorem{de}[thm]{Definition}
\newtheorem{ass}[thm]{Assumption}

\newcommand{\M}{{\mathbf{M}}}                  
\newcommand{\R}{{\mathbf{R}}}                  
\newcommand{\Rcond}[1]{{\mathbf{R}_{{#1}}}}                  
\newcommand{\tet}{{\operatorname{Tet}}}        
\newcommand{\fer}{{\operatorname{Fer}}}        

\newcommand{\gr}{{\Gamma}}                     
\newcommand{\den}{{\delta}}                    
\newcommand{\V}{{\mathrm{V}}}                  
\newcommand{\E}{{\mathrm{E}}}                  
\newcommand{\F}{{\mathrm{F}}}                
\newcommand{\D}[2]{\mathrm{D}({#2},{#1})}      
\newcommand{\dD}{\partial\D}
\newcommand{\dist}[2]{\mathrm{d}({#2},{#1})}   

\newcommand{\aut}[1][G]{{\operatorname{Aut}({#1})}}     
\newcommand{\autp}[1][G]{{\operatorname{Aut}^+({#1})}} 
\newcommand{\elt}{{g}}                                  

\newcommand{\primi}{{J}}                        
\newcommand{\primt}{{j}}                       
\newcommand{\rotv}{{S}}     
\newcommand{\rotf}{{R}}     
\newcommand{\id}{{\operatorname{Id}}}

\newcommand{\ZZ}{{\mathbb Z}}

\newcommand{\lowprime}{\raisebox{0.2ex}{$\scriptstyle\prime$}}
\newcommand{\Sym}{\textrm{Sym}}

\newcommand\pgfmathsinandcos[3]{%
  \pgfmathsetmacro#1{sin(#3)}%
  \pgfmathsetmacro#2{cos(#3)}%
}
\newcommand\LongitudePlane[3][current plane]{%
  \pgfmathsinandcos\sinEl\cosEl{#2} 
  \pgfmathsinandcos\sint\cost{#3} 
  \tikzset{#1/.estyle={cm={\cost,\sint*\sinEl,0,\cosEl,(0,0)}}}
}
\tikzset{%
  >=latex, 
  inner sep=0pt,%
  outer sep=2pt,%
  mark coordinate/.style={inner sep=0pt,outer sep=0pt,minimum size=3pt,
    fill=black,circle}%
}

\begin{document}
\title{Regular maps of high density}
\author[R.H.~Eggermont]{Rob H. Eggermont}
\address[Rob H. Eggermont]{
Department of Mathematics and Computer Science\\
Technische Universiteit Eindhoven\\
P.O. Box 513, 5600 MB Eindhoven, The Netherlands}
\email{r.h.eggermont@tue.nl}

\author[M.~Hendriks]{Maxim Hendriks}
\address[Maxim Hendriks]{Zermelo Roostermakers\\
Rhijngeesterstraatweg 40P, 2341 BV Oegstgeest, The Netherlands}
\email{mhendriks@zermelo.nl}
\maketitle
\tableofcontents

\begin{abstract} A regular map is a surface together with an embedded graph, having properties similar to those of the surface and graph of a platonic solid. We analyze regular maps with reflection symmetry and a graph of density strictly exceeding $\frac{1}{2}$, and we conclude that all regular maps of this type belong to a family of maps naturally defined on the Fermat curves $x^n+y^n+z^n=0$, excepting the one corresponding to the tetrahedron.
\end{abstract}

\input{introduction_plat_den}
\input{background_plat_den}
\input{density_theorem_plat_den}

\bibliographystyle{alpha}
\bibliography{bibliography_plat_den}

\end{document}

%% file: introduction_plat_den.tex

\section{Introduction}

Objects of high symmetry have been of interest for a long time. Thousands of years ago, the Greeks already studied the platonic solids, regular convex polyhedra with congruent faces of regular polygons such that the same number of faces meet at each vertex. The Greeks proved that there are only five of them: the tetrahedron, the cube, the octahedron, the dodecahedron, and the icosahedron.

One can consider the platonic solids to be orientable surfaces of genus $0$ with an embedded graph satisfying certain properties. The natural question is then whether this concept can be generalized. This leads to the notion of regular maps. Regular maps have also been studied for many years (see for example \cite{Siran2006}).

In contrast to the succinct and complete list of platonic solids, a complete classification for regular maps seems far off, but at least a few families of regular maps are known. Investigations into these objects is of great interest, amongst other reasons because of their intriguing connection to algebraic curves. Regular maps are special cases of \emph{dessins d'enfants} \cite{schneps94}. As such, the combinatorial structure of vertices, edges, and faces gives rise not only to a topological realisation, but even to a unique algebraic curve! Explicitly computing an ideal that defines some complex projective realisation of this algebraic curve for a given regular map is an ongoing area of research.


Two happy examples of a whole family of regular maps for which the algebraic curves are known, are naturally defined on the Fermat curves $x^n+y^n+z^n=0$. For a given Fermat curve one has to consider the action on this complex algebraic curve (and hence a real surface) of its group of algebraic automorphisms, which turn out to be realizable by linear maps in $\mathbb{C}^3$. A map presentation (see Section~\ref{sec:background}) of these maps was described concisely by Coxeter and Moser (see \cite{coxmos1980}), although they did not hint at the link to the Fermat curves, and were perhaps not aware of it. We give a more detailed description of the Fermat family in Section~\ref{sec:background}.

Regular maps can broadly be divided into two classes: chiral and reflexive. We will deal only incidentally with chiral maps. In low genus, all regular maps have been computed, with a list of reflexive maps up to genus $15$ appearing in \cite{condob2001} and with more recent lists at \cite{conderweb}, which runs up to genus $301$ for reflexive maps at the time of writing.

The combinatorial aspect of a regular map that is the focus of this paper is its (graph) density, defined in Section \ref{sec:background}.  One discovers from studying the maps of low genus, that having a high density is relatively rare. In fact, the only reflexive regular maps of low genus with simple graphs of density strictly exceeding $\frac{1}{2}$ are members of the Fermat family (every vertex in its graph being connected to precisely two thirds of the vertices), the sole exception being the regular map $\tet$ on the sphere, corresponding to the tetrahedron. This naturally leads to the question whether this is true in higher genus as well. This paper answers this question by classifying all reflexive regular maps with simple graphs of density strictly exceeding $\frac{1}{2}$.

\begin{thm}[Regular map density theorem]\label{thm:main} Let $\R$ be a reflexive regular map with simple graph of density strictly exceeding $\frac{1}{2}$. Then either $\R$ is $\tet$ or $\R$ is a member of the Fermat family.
\end{thm}

Before proving this theorem in Section~\ref{sec:mainthm}, we will give the definition of a regular map, show some properties of regular maps, and define the Fermat maps. 

%% file: background_plat_den.tex

\section{Background}\label{sec:background}
Let us now give a more formal definition of regular maps.

\begin{de} A \emph{map} $\M$ is an orientable surface $\Sigma$ together with an embedded connected finite graph $\gr$ with non-empty vertex set $\V$ and non-empty edge set $\E$ such that the complement of $\gr$ in $\M$ is a finite disjoint union of open discs. Each of these open discs is called a \emph{face} of $\gr$. If $v \in \V$, $e \in \E$ and $f$ is a face of $\gr$, we call any pair of these \emph{incident} if one of the pair is contained in the closure of the other.

A \emph{cellular homeomorphism} of $\M$ is a homeomorphism of $\M$ that induces a graph automorphism of $\gr$. An \emph{automorphism} of $\M$ is an equivalence class of cellular homeomorphisms under the equivalence relation of isotopy. We denote the set of automorphisms of $\M$ by $\aut[\M]$, and we write $\autp[\M]$ for the set of orientation-preserving automorphisms of $\M$. We call $\M$ \emph{chiral} if $\autp[\M] = \aut[\M]$ and we call $\M$ \emph{reflexive} otherwise. Below we will work refer to reflexive maps with the symbol $\R$ to stress that this property is assumed.

A map $\M$ is called \emph{regular} if for all directed $\overrightarrow{e},\overrightarrow{e\lowprime} \in \E$ there is an orientation-preserving automorphism of $\M$ mapping $\overrightarrow{e}$ to $\overrightarrow{e\lowprime}$.

\end{de}

\begin{ex} The platonic solids correspond to regular maps, simply by interpreting the traditional terms `vertices' and `faces' according to our definition above, and defining the embedded graph by rereading the term `edges'. We denote the regular map corresponding to the tetrahedron by $\tet$.
\end{ex}

\begin{no} If $\M$ is a map, we denote by $\Sigma(\M)$, $\gr(\M)$, $\V(\M)$ and $\E(\M)$ the corresponding surface, graph, vertex set and edge set. The set of faces of $\M$ is denoted $\F(\M)$.
\end{no}

Note that any cellular homeomorphism is uniquely determined up to cellular isotopy by the graph isomorphism it induces (although in general, not all graph isomorphisms are representable by a cellular homeomorphism).

\begin{lem}~\label{lem:autdefbyedge} Let $\M$ be a map. Suppose $\phi \in \autp[\M]$ fixes some directed edge $\overrightarrow{e}$. Then $\phi = \id$.
\end{lem}

\begin{proof} Suppose $e'$ is an edge such that $e$ and $e'$ share a common vertex. Let $v$ be a vertex incident to both $e$ and $e'$ and consider a local picture around $v$. The only way for $\phi$ to fix $\overrightarrow{e}$ and to preserve orientation is to act as the identity locally. In particular, this means it must fix $\overrightarrow{e\lowprime}$ as well. Using connectedness of $\gr(\M)$, it follows that $\phi$ fixes all directed edges and hence is the identity as a graph isomorphism. This means $\phi = \id$.
\end{proof}

\noindent
As a direct corollary, any orientation-preserving automorphism of a map is uniquely determined by the image of a single directed edge. Similarly, any automorphism of a map is uniquely determined by the combination of its orientation and the image of a single directed edge.

\begin{no} Let $\M$ be a regular map. Let $f$ be a face of $\M$ with counterclockwise orientation. The boundary of $f$ is a union of edges, and the number of edges (where an edge is counted with multiplicity $2$ if it borders on the same face twice) in such a union does not depend on the face because $\M$ is regular. We always use the letter $p$ to denote this number. Let $e_1,e_2,\ldots,e_p$ be the edges on the boundary of $f$ (possibly containing doubles) in counterclockwise order. Orient these edges in a way compatible with the orientation of $f$. There is a unique orientation-preserving automorphism of $\M$ mapping $\overrightarrow{e_1}$ to $\overrightarrow{e_2}$. By necessity, it fixes $f$ and maps $\overrightarrow{e_i}$ to $\overrightarrow{e_{i+1}}$, taking indices modulo $p$ if necessary. Effectively, it can be seen as a rotation around $f$. We denote it by $\rotf_f$. Clearly, it has order $p$. Observe that any orientation-preserving automorphism of $\M$ fixing $f$ is a power of $\rotf_f$.

Similarly, given a vertex $v$ of $\M$, the number of edges incident to $v$ does not depend on the vertex. We always use the letter $q$ to denote this number. Let $e_1,e_2,\ldots,e_q$ be the edges incident to $v$ (possibly containing doubles) in counterclockwise order. Orient these edges locally as arrows departing from $v$. There is a unique orientation-preserving automorphism of $\M$ mapping $\overrightarrow{e_1}$ to $\overrightarrow{e_2}$. It fixes $v$ and also maps $e_i$ to $e_{i+1}$, so it can be seen as a rotation around $v$. We denote it by $\rotv_v$. Clearly, it has order $q$. Observe that any orientation-preserving automorphism of $\M$ fixing $v$ is a power of $\rotv_v$.
\end{no}

Let $\M$ be a reflexive regular map, and let $e$ be an edge of $\M$. We show that there are two orientation-reversing automorphisms that fix $e$ as a non-directed edge. We will call these reflections.

Let $\sigma \in \aut[\M]\setminus\autp[\M]$. There is an automorphism $g \in \autp[\M]$ such that $g(\sigma(\overrightarrow{e})) = \overrightarrow{e}$ by definition of a regular map. Moreover, there is $h \in \autp[\M]$ that reverses $e$ as a directed edge. The automorphisms $g\sigma$ and $hg\sigma$ do not preserve orientation and are two distinct automorphisms that fix $e$ as a non-directed edge. The first one fixes $\overrightarrow{e}$ as a directed edge. This shows existence. For uniqueness, if $\sigma,\sigma'$ do not preserve orientation and if both either fix $\overrightarrow{e}$ or reverse it, then $\sigma^{-1}\sigma'$ is orientation-preserving and fixes $\overrightarrow{e}$, and hence is $\id$ by Lemma~\ref{lem:autdefbyedge}.

A visualization of rotations and reflections can be seen in Figure~\ref{Fig:Basic_rotations}.

\begin{figure}[h]
\includegraphics[width=12cm]{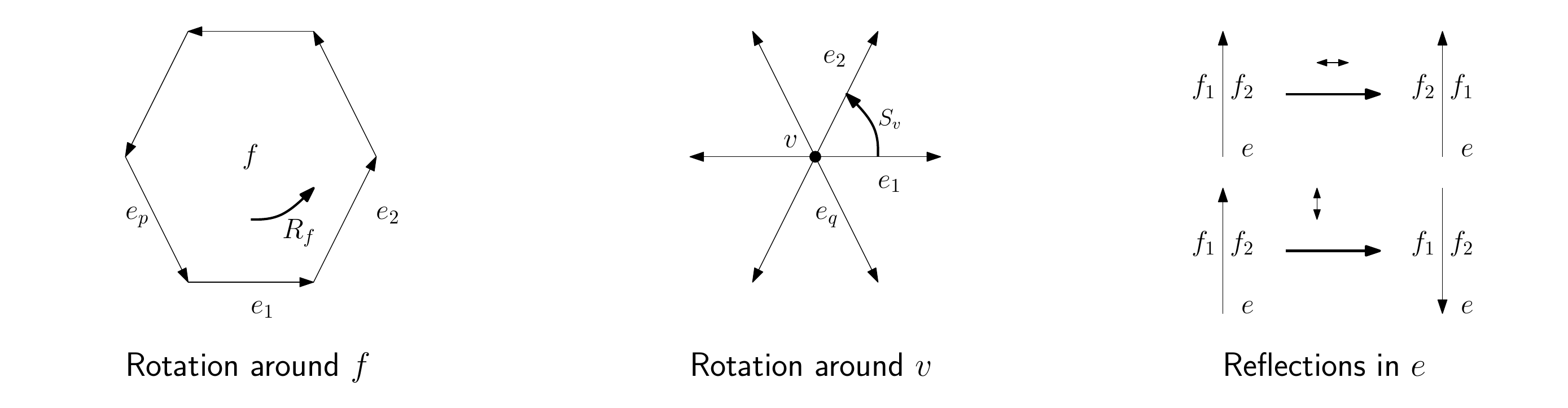}
\caption{Rotations and reflections. An indication of the action is given by curved arrows in the left and middle figure (rotations), and by the two-way arrows in the right figure (reflection).
}\label{Fig:Basic_rotations}
\end{figure}

As an example, suppose $v$ and $w$ are adjacent vertices, and after fixing an orientation, let $f$ be the face to the left of the oriented edge $\overrightarrow{(v,w)}$. The rotation $\rotv_v$ around $v$ maps $\overrightarrow{(v,w)}$ to $\overrightarrow{(v,\rotv_v(w))}$, and the rotation $\rotf_f$ around $f$ maps the latter to $\overrightarrow{(w,v)}$. This means the automorphism $\rotf_f\circ \rotv_v$ inverts the oriented edge $\overrightarrow{(v,w)}$. As a consequence, it has order $2$. In particular, for any regular map $\M$, if $f$ is a face adjacent to a vertex $v$, then $(\rotf_f\circ \rotv_v)^2 = \id$.

\begin{lem} Let $\M$ be a regular map, let $v$ be a vertex of $\M$, and let $f$ be a face of $\M$ incident to $v$. Then $\rotv_v$ and $\rotf_f$ generate $\autp[\M]$.
\end{lem}

This lemma can be proved by showing that $\rotv_v$ and $\rotf_f$ can map a given directed edge to any other directed edge. A proof can be found in \cite[Lemma 1.1.6]{hendriks2013}. For convenience, when we have $v \in \V(\M)$ and $f$ a face of $\M$ incident to $v$, we occasionally use $\rotv$ and $\rotf$ instead of $\rotv_v$ and $\rotf_f$. An important consequence of the lemma is that we can pin down a regular map $\M$ without explicitly talking about its topology. We can give a presentation of $\autp[\M]$ in the generator pair $(\rotf,\rotv)$. We call this a \emph{standard map presentation} of $\M$.


\begin{ex} We have $|\autp[\tet]| = 12$, the cardinality of the set of directed edges of $\tet$. Both the number of edges incident to a vertex and the number of edges incident to a face are $3$. This means $q = 3$ and $p = 3$. Let $v \in \V(\tet)$ and $f$ a face incident to $v$. We find the obvious relations $\rotv^3 = \id$ and $\rotf^3 = \id$. Moreover, we have the relation $(\rotv\rotf)^2 = 1$, from the fact that $\rotv\rotf$ reverses the edge $(v,\rotf^{-1}(v))$. It turns these are the only relators necessary to define a standard map presentation: $\autp[\tet] = \langle \rotf, \rotv \mid \rotf^3, \rotv^3, (\rotf\rotv)^2 \rangle$.
\end{ex}

\begin{re}
A regular map is completely determined by a standard map presentation. Knowing just the isomorphism type of $\autp[\M]$ is insufficient, however. For example, the regular maps $\Rcond{3.1}$ and $\Rcond{10.9}$ both have $\autp[\M] \cong \textrm{PSL}(2,7)$. Even fixing the triple $(\autp[\M],p,q)$, and thereby also the genus, does not necessarily determine a unique regular map. The counterexamples, \emph{tuplets}, are of special interest (a different story altogether). The first examples occur in genus $8$ (the twins $\Rcond{8.1}$ and $\Rcond{8.2}$) and genus $14$ (the first Hurwitz triplet, $\Rcond{14.1}$, $\Rcond{14.2}$, $\Rcond{14.3}$). The subscripts indicate the precise maps as listed on \cite{conderweb}.
\end{re}

A small theory of $\autp[\M]$-equivariant cellular morphisms between regular maps can be developed, as described in \cite[Section 1.6]{hendriks2013}. The notion happily coincides with taking certain quotients of $\autp[\M]$, and a result we will use later in Lemma~\ref{j=q->q=2} and Proposition~\ref{dens2/3} is the following.
\vspace{\baselineskip}

\begin{prop}~\label{prop:modding}
Suppose $\M$ is a regular map and $H$ a normal subgroup of $\aut[\M]$ that is contained in $\autp[\M]$ and does not contain an automorphism that reverses some edge of $\M$. Then $\aut[\M]/H$ is the automorphism group of a regular map $\overline{\M}$ satisfying $\gr(\overline{\M}) = \gr(\M)/H$ and $\F(\overline{\M}) = \F(\M)/H$. There is a branched cellular covering $\M \to \overline{\M}$ with the fiber of a cell of $\overline{\M}$ a coset of $H$. Each cell of $\M$ contains at most one ramification point, and each cell of $\overline{\M}$ at most one branch point. These numbers only depend on the dimension of the cell.
\end{prop}

\noindent
Here, by $\gr(\M)/H$ we mean the graph obtained by identifying vertices respectively edges of $\gr(\M)$ if they lie in the same $H$-orbit, and by $\F(\M)/H$ we mean the set obtained by identifying faces of $\F(\M)$ if they lie in the same $H$-orbit.

%

\begin{ex}[Fermat maps] For $n \in \mathbb{Z}_{>0}$, let
$$
G_n = \langle\rotf,\rotv \mid \rotf^3, \rotv^{2n}, (\rotf\rotv)^2, [\rotf,\rotv]^3 \rangle .
$$
For each $n$, the group $G_n$ is the group of orientation-preserving automorphisms of a regular map that we call the Fermat map $\fer(n)$, obtained by considering the solutions of $x^n+y^n+z^n = 0$, acted upon by its algebraic automorphism group. We omit the proof for this claim, but we do note that $\fer(n)$ is a reflexive regular map, and that this is a remarkable property. The group structure can be described as $G_n \cong \ZZ_n^2 \rtimes \Sym_3$. The graph $\gr(\fer(n))$ is a simple graph, has $3n$ vertices and is of genus ${n-1 \choose 2}$. Each of the faces of $\fer(n)$ is a triangle, since $\rotf$ has order $3$.
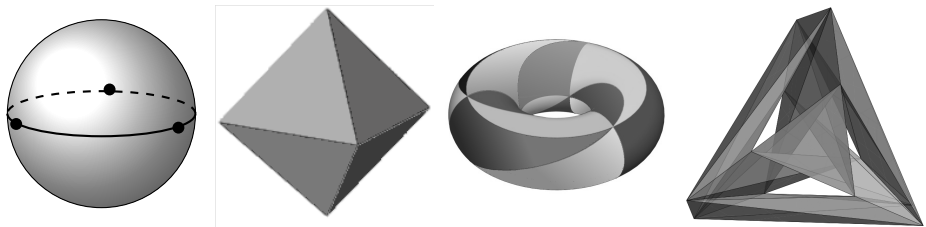
\begin{figure}[ht!]
\centering
\raisebox{3mm}{
\begin{tikzpicture}[scale=0.5]
\def\R{2.5} 
\def\angEl{15} 
\tikzset{
    partial ellipse/.style args={#1:#2:#3}{
        insert path={+ (#1:#3) arc (#1:#2:#3)}
    }
}
\filldraw[ball color=white] (0,0) circle (\R);
\draw[thick, dashed] (0,0) [partial ellipse=0:180:25mm and 6mm];
\draw[thick] (0,0) [partial ellipse=180:360:25mm and 6mm];
\foreach \t in {85, 205, 325} {
\LongitudePlane[pzplane]{\angEl}{\t};
 \path[pzplane] (0:\R) coordinate[mark coordinate, scale=1.5] ();
}
\end{tikzpicture}
}
\includegraphics[width=0.23\textwidth]{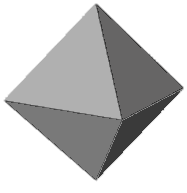}
\raisebox{5mm}{\includegraphics[width=0.24\textwidth]{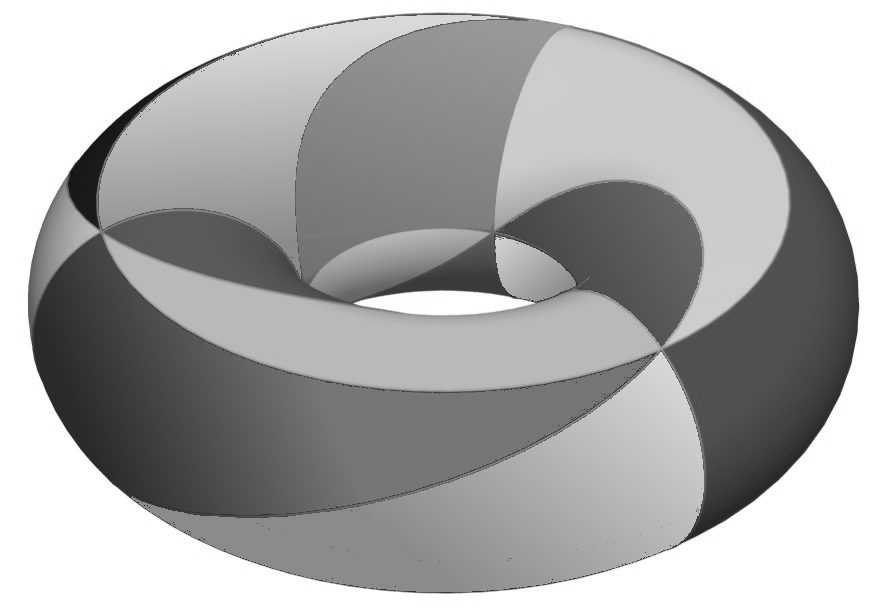}}
\includegraphics[width=0.26\textwidth]{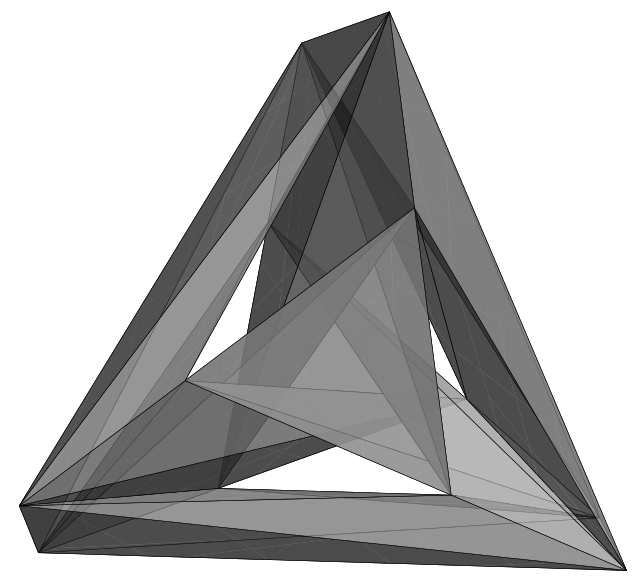}
\caption{The first Fermat maps ($n=1,2,3,4$). The visualisation of $\fer(4)$ was constructed in 1987 by Ulrich Brehm \cite{Bre87}.}\label{fig:fermatmaps}
\end{figure}

\end{ex}

\begin{de} Let $\M$ be a regular map and let $v,v' \in \V(\M)$. Then the \emph{distance} between $v$ and $v'$ is the minimal length of a path in $\gr(\M)$ from $v$ to $v'$; it is denoted by $\dist{v'}{v}$. The set of vertices at distance at most $i$ from $v$ is denoted $\D{i}{v}$. The set of vertices at distance precisely $i$ from $v$ is denoted by $\dD{i}{v}$.

The \emph{density} $\den(\M)$ of $\M$, which is the central notion for the rest of this paper, is defined as
$$
\den(\M) := \frac{|\dD{1}{v}|}{|\V(\M)|} .
$$
\end{de}

\noindent
Note that $\den(\M)$ does not depend on the choice of $v$ by the regularity of the map.

\begin{ex} The Fermat maps all have density $\frac{2}{3}$. The tetrahedron $\tet$ has density $\frac{3}{4}$.
\end{ex}

\noindent
Our main objective from here on will be to show that the Fermat maps and the tetrahedron are the only reflexive regular maps with simple graph of density strictly exceeding $\frac{1}{2}$, as announced in our \emph{regular map density theorem} \ref{thm:main}.


%% file: density_theorem_plat_den.tex

\section{The Regular Map Density Theorem}\label{sec:mainthm}
Let us write $\V = \V(\R)$ for convenience. We start with a rather technical but crucial lemma:

\begin{lem}\label{regularity} Let $v,v' \in \V$. Suppose we have $j \in \ZZ$ such that $\rotv_{v}^j$ fixes $v'$. Then the following claims hold.
\begin{description}
\item[1] There is $k \in \ZZ$ such that $\rotv_{v}^j = \rotv_{v'}^{kj}$. Moreover, $k$ is well-defined and invertible modulo $\frac{q}{\gcd(j,q)}$.
\item[2] Let $\elt \in \aut[\R]$. Suppose $\rotv_{v}^j = \rotv_{v'}^{kj}$. Then $\rotv_{\elt(v)}^j = \rotv_{\elt(v')}^{kj}$. Moreover, $\rotv_{v}^j$ fixes $\elt(v)$ if and only if it fixes $\elt(v')$ and in this case, if $\rotv_{v}^j = \rotv_{\elt(v)}^{l}$, then $\rotv_{v}^j = \rotv_{\elt(v')}^{kl}$.
\item[3] Let $k \in \ZZ$ such that $\rotv_{v}^j = \rotv_{v'}^{kj}$. For all $i \in \ZZ$, we have $\rotv_{v}^j = \rotv_{\rotv_{v}^i(v')}^{kj}$ and $\rotv_{v}^j = \rotv_{\rotv_{v'}^i(v)}^j$.
\item[4] Let $\elt_1,\ldots,\elt_n \in \aut[\R]$ and suppose there is $v_i \in \V$ such that $\rotv_{v}^j$ fixes both $v_i$ and $\elt_i(v_i)$ for all $i \in \{1,2,\ldots,n\}$. Then $\rotv_{v}^j$ fixes $\elt_n\elt_{n-1}\ldots\elt_1(v')$.
\end{description}
\end{lem}

\begin{proof} For Claim $\mathbf{1}$, note that any orientation-preserving element that fixes $v'$ is a rotation around $v'$, say $\rotv_{v}^j = \rotv_{v'}^i$. The order of these rotations must be equal, and in this case is equal to the smallest positive integer $x$ such that $xj \equiv 0\mod q$, which is $\frac{q}{\gcd(j,q)}$, using the fact that $(\rotv_{v}^j)^x = 1$ precisely if $xj$ is a multiple of $q$. Analogously, we observe that the order of these rotations must be equal to $\frac{q}{\gcd(i,q)}$, meaning $\gcd(i,q) = \gcd(j,q)$. In particular, both $i$ and $j$ are integer multiples of $\gcd(j,q)$, and moreover, these multiples must be invertible modulo $\frac{q}{\gcd(j,q)}$ (if not, the order of $i$ or $j$ modulo $q$ would be strictly smaller than $\frac{q}{\gcd(j,q)}$). Clearly, there exists $k$, well-defined and invertible modulo $\frac{q}{\gcd(j,q)}$, such that $jk \equiv i$ mod $q$. This shows Claim $\mathbf{1}$.

For Claim $\mathbf{2}$, we have the equalities $\rotv_{\elt(v)}^{j} = \elt\rotv_{v}^{\pm j}\elt^{-1} = \elt\rotv_{v'}^{\pm kj}\elt^{-1} = \rotv_{\elt(v')}^{kj}$, where the sign in the exponent corresponds to the orientation of $\elt$ being positive or negative. If $\rotv_{v}^j$ fixes $\elt(v)$ respectively $\elt(v')$, it is a power of $\rotv_{\elt(v)}^j$ respectively $\rotv_{\elt(v')}^j$ (which is itself a power of $\rotv_{\elt(v)}^j$). In either case, $\rotv_{v}^j$ fixes both $\elt(v)$ and $\elt(v')$. Moreover, if $\rotv_{v}^j = \rotv_{\elt(v)}^l$, we have $\rotv_{v}^j = \rotv_{\elt(v')}^{kl}$. This completes the proof of Claim $\mathbf{2}$.

The equality $\rotv_{v}^j = \rotv_{\rotv_{v}^i(v')}^{kj}$ is easily seen because of Claim $\mathbf{2}$, using $\elt = \rotv_{v}^i$. The equality $\rotv_{v}^j = \rotv_{\rotv_{v'}^i(v)}^j$ follows by symmetry. This shows Claim $\mathbf{3}$.

To prove Claim $\mathbf{4}$, note that by Claim $\mathbf{2}$, we have $\rotv_v^j$ fixes $\elt_1(v')$ if and only if it fixes $\elt_1(v)$. Exchanging the roles of $v'$ and $v_1$, we have $\rotv_v^j$ fixes $\elt_1(v)$ if and only if it fixes $\elt_1(v_1)$. The latter is true by assumption, so $\rotv_v^j$ fixes $\elt_1(v')$. Claim $\mathbf{4}$ now follows by induction.
\end{proof}

While we formulated the previous lemma rather technically, the statements should be seen in a more geometrical and intuitive way. Any rotation $\rotv_{v_0}^j$ that fixes a vertex $v_1$ must be a rotation around $v_1$ of the same order (Claim $1$). Moreover, such a situation translates well under conjugation (Claims $2$ and $3$). Finally, the set of fixed points of an orientation-preserving automorphism $\rotv_v^j$ is closed under the set of automorphisms $\elt'$ that map at least one fixed point of $\elt$ to another fixed point of $\elt$ (Claim $4$). We will use this lemma often.

From here on, we will work under the following assumption.
\begin{ass}
$\R$ is a reflexive regular map with a simple graph.
\end{ass}

An easy graph lemma gives us a point of departure to say something about regular maps with high density.

\begin{lem}\label{diamleq2} If $\gr(\R)$ has density $\den(\R) \geq \frac{1}{2}$, then $\gr(\R)$ has diameter at most $2$.
\end{lem}

\begin{proof} Let $v_0,v_1 \in \V$. Suppose that $\dist{v_0}{v_1}\geq 2$. Then $\dD{1}{v_0}$ and $\dD{1}{v_1}$ are contained in $V - \{v_0,v_1\}$, which has cardinality $|V|-2$. Since we have $|\dD{1}{v_0}| + |\dD{1}{v_1}| \geq |V| > |V|-2$, we find that $v_0$ and $v_1$ share at least two common neighbors. In particular, this implies $\dist{v_1}{v_0} \leq 2$.
\end{proof}

In the following proposition, we will show that the faces of reflexive regular maps of high density are triangles ($p=3$).

\begin{prop}\label{Prop:p=3} Suppose $\den(\R) > \frac{1}{2}$. Then $p=3$.
\end{prop}

\begin{proof} Let $v \in \V$ and let $f$ be a face incident to $v$. Consider $v_1 =\rotf_{f}^2(v)$ and suppose $v_1 \not\in\dD{1}{v}$. Then the same holds for $\rotv_{v}^i(v_1)$ for any $i \in \{0,1,\ldots,q-1\}$, because it preserves distances. In particular, the size of the orbit under $\langle\rotv_v\rangle$ of $v_1$ is at most $|\V \setminus \dD{1}{v}|$. We have $|\V| < 2q$ because $\den(\R) > \frac{1}{2}$. Therefore, $|\V \setminus \dD{1}{v}| = |V|-q < 2q-q =q$. So the size of the orbit of $v_1$ under $\langle\rotv_v\rangle$ is strictly smaller than $q$, and hence there is $j \in \{1,\ldots,q-1\}$ satisfying $\rotv_{v}^j(v_1) = v_1$.

Note that on one hand, $\rotv_{v}^j$ does not fix any neighbor of $v$, and since $\den(\R) > \frac{1}{2}$, we find $\rotv_{v}^j$ has at most $|V|-q < q$ fixed points. On the other hand, observe that for $v' = \rotf_{f}(v)$, we have $v_1 = \rotv_{v'}^{-1}(v)$. By Claim $4$ of Lemma \ref{regularity}, using $g_i = \rotv_{v'}^{-1}$, we find $\rotv_{v}^j$ fixes $\rotv_{v'}^i(v)$ for any $i$, so $\rotv_{v}^j$ fixes all neighbors of $v'$, meaning it has at least $q$ fixed points, a contradiction. We conclude $v_1 \in \dD{1}{v}$.

Label the neighbors of $v$ by the elements of $\ZZ/q\ZZ$, counterclockwise and let $f_i$ be the face on the left of $\overrightarrow{(v,i)}$ (consistent with the orientation).

Suppose $\rotv_{0}^{-1}(v) = i$. Since $\R$ is reflexive and the reflection that fixes the oriented edge $\overrightarrow{(v,0)}$ maps $i$ to $-i$, we find $\rotv_{0}(v)=-i$. Observe that this means that $\rotv_{i}(v) = \rotv_{v}^i\rotv_{0}\rotv_{v}^{-i}(v) = 0$. We now see two faces on the left of $\overrightarrow{(0,i)}$, being $f_0$ and $f_{i-1}$. We conclude $f_0 = f_{i-1}$. Since $v$ occurs only once on each face (using the assumption that $\gr(\R)$ is simple), we conclude $i=1$ and hence $p=3$, as was to be shown.
\end{proof}

\begin{figure}[h]
\includegraphics[width=12cm]{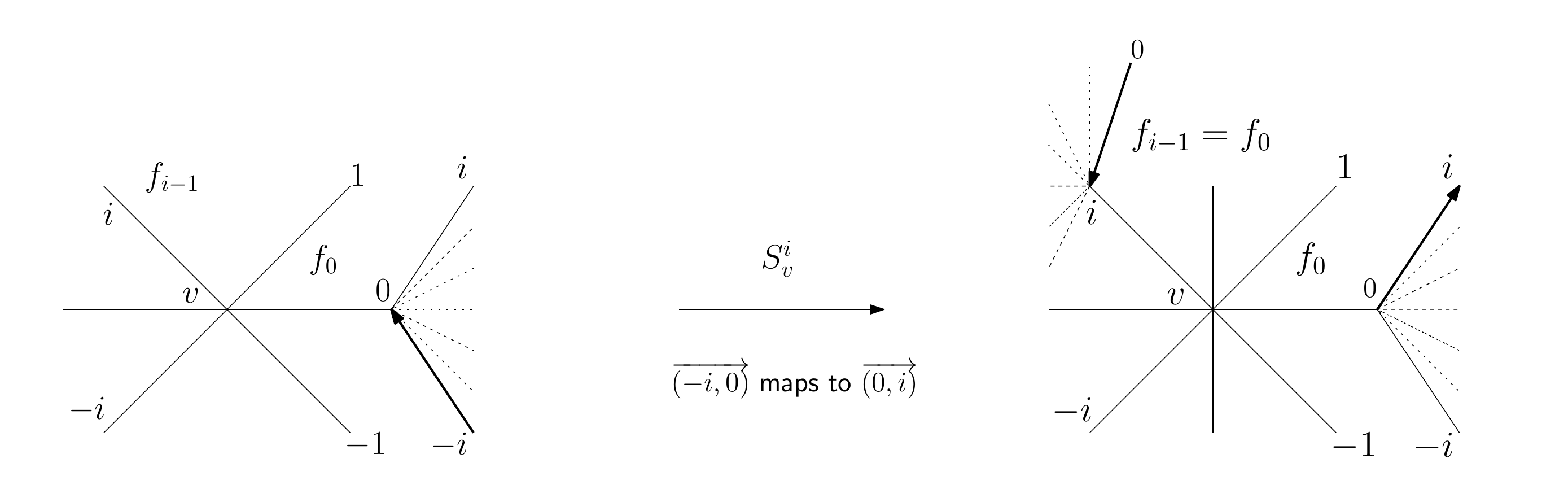}
\caption{Proof of Proposition~\ref{Prop:p=3}}\label{Fig:p=3}
\end{figure}

\begin{de} Assume we have a regular map with $p = 3$, and let $v \in \V$. A \textit{diagonal neighbor} of $v$ is an element of $\V$ of the form $\rotv_{w}^2(v)$ with $w$ a neighbor of $v$. Let $v' \in \V$ as well. We call $v$ and $v'$ \textit{diagonally aligned} if there is a sequence $(v_0,v_1,\ldots,v_n)$ of elements of $\V$ satisfying $v_0 = v$, $v_n = v'$ and for all $i \in \{1,\ldots,n\}$ the vertex $v_i$ is a diagonal neighbor of $v_{i-1}$.
\end{de}

If $v \in \V$ is part of a triangle $vwu$, then $wu$ is part of precisely one other triangle, say $wuv'$. In this case, $v'$ is a diagonal neighbor of $v$. All diagonal neighbors of $v$ are of this form.

\begin{figure}[h]
\includegraphics[width=12cm]{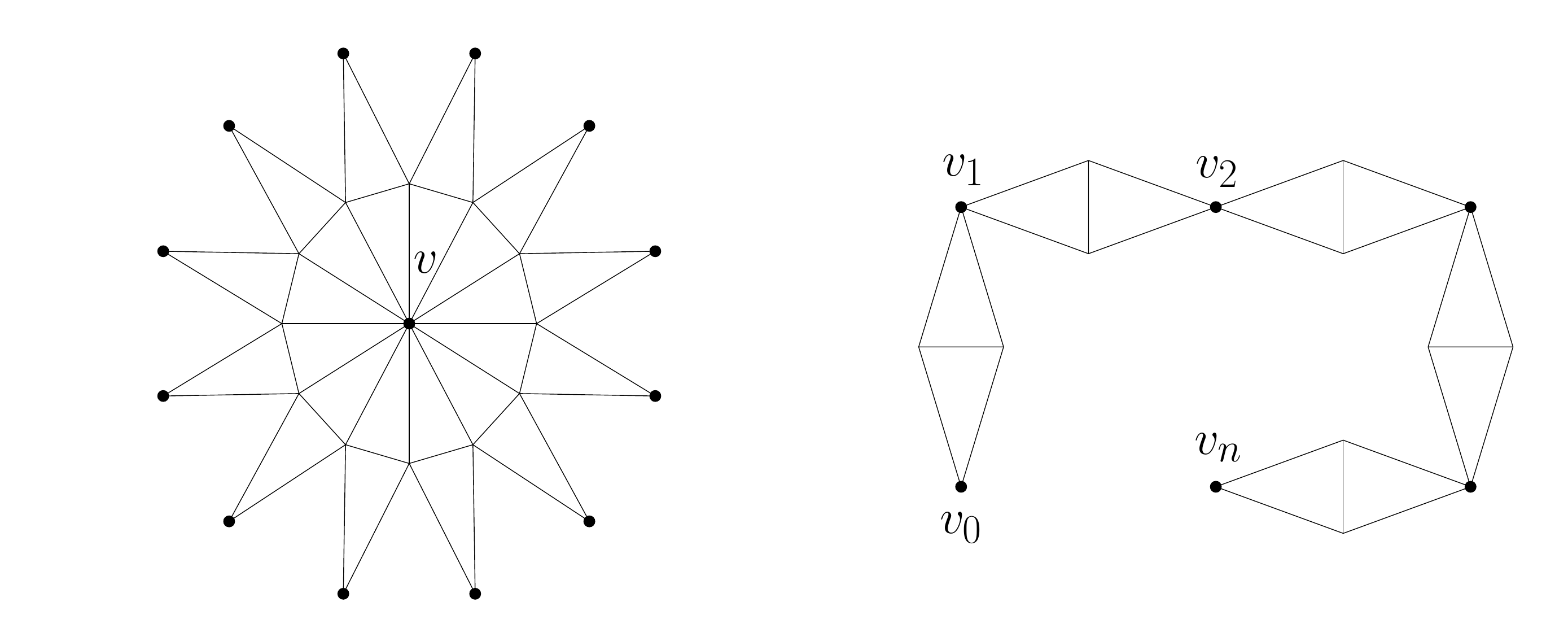}
\caption{Diagonal neighbors and diagonal alignment}\label{Fig:diag_neighb}
\end{figure}

The observation that $v'$ is a diagonal neighbor of $v$ if and only if $v$ is a diagonal neighbor of $v'$ shows that being diagonally aligned is an equivalence relation. We write $v\sim v'$ if $v$ and $v'$ are diagonally aligned. It is easy to see that being diagonally aligned is preserved by graph isomorphisms.

For $v \in \V$, we write $\V_v = \{v' \in \V: v' \sim v\}$. It is an easy exercise to see that if $vwu$ is a face of $\gr(\V)$, then any element of $\V$ is diagonally aligned to at least one of $v$, $w$ or $u$ (and in fact, either $\V_v$, $\V_w$ and $\V_u$ are pairwise disjoint or $\V_v = \V_w = \V_u = \V$).

\begin{de}
Let $\R$ be a regular map. We define $\primi$ to be the minimal element of $\{1,2,\ldots,q\}$ such that $\rotv_v^{\primi}$ fixes all elements in $\V_v$. We call $\primi$ the \emph{primitive period} of $\R$.
\end{de}

Note that $\primi$ exists because $\rotv_v^q = \id$. Moreover, $\primi$ is independent of choice of $v$ by Claim $2$ of Lemma~\ref{regularity} (using the fact that $v \sim v'$ if and only if $g(v) \sim g(v')$ for all $v,v' \in \V$ and all $g \in \aut[\R]$), so we are justified in not using a subscript. Thirdly, $\primi$ divides $q$, since $\rotv_v^q = \id$ fixes $\V_v$ pointwise.

\begin{ass}
From this point on, we add to our previous assumptions ($\R$ is a reflexive regular map with simple graph) that $\R$ satisfies $\den(\R) > \frac{1}{2}$. In particular, we will have $p = 3$ and $q \geq 2$.
\end{ass}

Let $v \in \V$. The following lemma shows that we can classify $\V_v$ as the set of fixed points of $\rotv_v^{\primi}$.

\begin{lem}\label{lem:diagJ}
Let $v \in \V$ and let $v'$ be a diagonal neighbor of $v$. Then the following claims hold.
\begin{description}
\item[1] The element $\primi$ is the minimal element in $\{1,\ldots,q\}$ such that $\rotv_v^{\primi}$ fixes $v'$.
\item[2] We have $\primi = q$ if and only if $v' \in \dD{1}{v}$.
\item[3] We have $\primi = q$ if and only if $\V_v = \V$.
\item[4] If $w \in \V$ is fixed by $\rotv_v^{\primi}$, then $w \in \V_v$.
\item[5] If $\primi < q$, then $\V_v \cap \dD{1}{v} = \emptyset$, $q$ is even, and $\den(\R) \leq \frac{2}{3}$.
\item[6] If $\primi = q$, then $q$ is odd.
\end{description}
\end{lem}

\begin{proof} Claim $\mathbf{1}$ can be shown by induction. Suppose there is $\primi'$ such that $\rotv_v^{\primi'}$ fixes $v'$. Then by Claim $3$ of Lemma~\ref{regularity}, it fixes any diagonal neighbor of $v$ (since all of these are of the form $\rotv_{v}^i(v')$), and analogously it fixes any diagonal neighbor of any diagonal neighbor of $v$, etcetera. Hence $\rotv_v^{\primi'}$ fixes $\V_v$, and hence $\primi$ divides $\primi'$. Clearly, $\rotv_v^{\primi}$ fixes $v'$, so this shows Claim $\mathbf{1}$.

For Claim $\mathbf{2}$, note that $\rotv_v$ induces bijections on the sets $\{v\}$, $\dD{1}{v}$ and $\dD{2}{v}$. Note that $\dD{1}{v}$ has cardinality $q$ and both $\{v\}$ and $\dD{2}{v}$ have cardinality less than $q$ since $\den(\R) > \frac{1}{2}$. In particular, if $v' \not \in \dD{1}{v}$, its orbit under the powers of $\rotv_v$ has cardinality strictly less than $q$, and hence $\primi < q$. Conversely, if $v' \in \dD{1}{v}$, then it is fixed by $\rotv_v^{\primi}$, and therefore $\primi = q$.

For Claim $\mathbf{3}$, note that if $\V_v = \V$, then it contains a neighbor of $v$, and hence $\primi = q$. Conversely, if $\primi = q$, then $v'$ is a neighbor of $v$ by the second part. Since $\V_v$ is the equivalence class of $v$ under $\sim$, any rotation around $v$ fixes it as a set. Hence all neighbors of $v$ are elements of $\V_v$. But then $|\V_v| \geq q > \frac{1}{2}|\V|$, and hence $\V_v = \V$ (using the fact that either $|\V_v| = \frac{1}{3}|\V|$ or $\V_v = \V$).

For Claim $\mathbf{4}$, observe that if $\rotv_v^{\primi}$ fixes $w$, then it also fixes $\V_w$. If $\V_w \neq \V_v$, then $\V_w$ contains a neighbor of $v$, and hence $\primi = q$. This gives a contradiction, since by Claim $\mathbf{3}$, $\V_v = \V$ and hence $w \in \V_v$.

For Claim $\mathbf{5}$, suppose $\primi < q$. If $\V_v \cap \dD{1}{v} \neq \emptyset$, then $\rotv_v^{\primi}$ fixes a neighbor of $v$, which gives a contradiction. So $\V_v \cap \dD{1}{v} = \emptyset$.

Let $w$ be a neighbor of $v$ and consider the set $\{\rotv_w^{2i}(v): i \in \{0,1,\ldots,\lfloor \frac{q-1}{2} \rfloor\}\}$ of cardinality $1+\lfloor \frac{q-1}{2} \rfloor\}$. All of these elements lie in $\V_v$. If $q$ is odd, then this set contains $\rotv_w^{q-1}(v)$, which is a neighbor of $v$, a contradiction. Hence $q$ is even, and this set has cardinality $\frac{q}{2}$. In particular, we find $|\V \setminus \dD{1}{v}| \geq \frac{q}{2}$, and hence $\den(\R) \leq \frac{q}{q + \frac{q}{2}} = \frac{2}{3}$.

For Claim $\mathbf{6}$, suppose $\primi = q$ and $q$ is even. Then $v' \in \dD{1}{v}$. Let $w \in \dD{1}{v}$ such that $v' = \rotv_w^2(v)$, and note that there is a directed edge $\overrightarrow{(w,\rotv_w(v))}$. Consider the (unique) reflection that maps this edge to $\overrightarrow{(\rotv_w(v),w)}$. Since $q$ is even, this reflection does not fix any neighbor of $v$. On the other hand, it fixes $v'$, a contradiction.
\end{proof}

We can now show that the only reflexive regular map of density greater than $\frac{1}{2}$ with $q$ odd is $\tet$. After showing this, we can focus on the case where $q$ is even.

\begin{prop}\label{classes} Suppose $\den(\R) > \frac{1}{2}$ and $q$ is odd. Then $\R$ is $\tet$, and $\den(\R) = \frac{3}{4}$.
\end{prop}

\begin{proof} Let $v \in \V$. By Lemma~\ref{lem:diagJ}, the primitive period $\primi$ of $\R$ is equal to $q$. Number the neighbors of $v$ by $0,1,\ldots,q-1$ clockwise and let $v' = \rotv_{0}^2(v)$, a diagonal neighbor of $v$. It is a neighbor of $v$ by the second part of Lemma~\ref{lem:diagJ}.

Now, consider the reflection that maps that maps $\overrightarrow{(0,q-1)}$ to $\overrightarrow{(q-1,0)}$. It fixes $v$ and $v'$. Moreover, it maps
any neighbor $i$ of $v$ to $-i-1 \mod q$. We conclude $v' = \frac{q-1}{2}$.

Observe that we have a face $0,q-1,\frac{q-1}{2}$. Applying the rotation $\rotv_{v}^{\frac{q+1}{2}}$ yields the triangle $\frac{q+1}{2},\frac{q-1}{2},0$. Because rotations preserve orientation, we conclude $\frac{q+1}{2} \equiv q-1 \mod q$. Since $q > 2$, we conclude $\frac{q+1}{2} = q-1$ and hence $q = 3$. We now easily deduce $\R = \tet$.
\end{proof}

\begin{ass} From here on, we'll assume $\R$ is a reflexive regular map with simple graph satisfying $\den(\R) > \frac{1}{2}$ that is not equal to $\tet$. In particular, we will have $p = 3$, $q$ is even, and $\den(\R) \leq \frac{2}{3}$. Moreover, we have $\primi < q$, and for all $v$, no neighbor of $v$ is diagonally aligned to $v$.
\end{ass}

Suppose $vwu$ is a face of $\gr(\R)$. Then $\rotv_v$ maps $\V_w$ to $\V_u$ and vice versa. In particular, $\rotv_v^2$ is a bijection of $\V_w$ and $\V_u$ (and of $\V_v$ of course). This motivates us to give the following definition:

\begin{de}
Let $\primt := \mathrm{lcm}(\primi,2)$. We call $\primt$ the \emph{even period} of $\R$.
\end{de}

\begin{lem}\label{lem:primtand4commute}
For any $v,w \in \V$, we have $[\rotv_v^{\primt},\rotv_w^{\primt}] = \id$ and $[\rotv_v^{\primt},\rotv_w^{4}] = \id$.
\end{lem}

\begin{proof} If $v\sim w$, we have $\rotv_v^{\primt}$ is a rotation around $w$, and hence it commutes with $\rotv_w$, which implies both of the relations.

Suppose $v\not\sim w$. Note that $\rotv_v^{\primt}$ fixes $\V_v$ pointwise and hence for any $v' \in \V_v$, we have $[\rotv_v^{\primt},\rotv_w^{\primt}](v') = \rotv_v^{\primt}\rotv_w^{\primt}\rotv_v^{-\primt}\rotv_w^{-\primt}(v') = \rotv_w^{\primt}\rotv_w^{-\primt}(v') = v'$, using the fact that $\rotv_w^{\primt}$ fixes $\V_v$ as a set and $\rotv_v^{\primt}$ fixes $\V_v$ pointwise. Likewise, $[\rotv_v^{\primt},\rotv_w^{\primt}]$ fixes $\V_w$ pointwise. There is $v' \in \V_v$ such that $v'$ is a neighbor of $w$. Now $[\rotv_v^{\primt},\rotv_w^{\primt}]$ fixes the edge $\overrightarrow{(v',w)}$ and hence it is $\id$.

Also, we have $[\rotv_v^{\primt},\rotv_w^{4}] = \rotv_v^{\primt}\rotv_{\rotv_w^4(v)}^{-\primt}$. Let $v'' = \rotv_w^2(v')$. We see that both $v'$ and $\rotv_w^4(v')$ are diagonal neighbors of $v''$. There is $k$ such that $\rotv_{v'}^j = \rotv_{v''}^{kj}$ by Claim $1$ of Lemma~\ref{regularity}. Note that $k$ only depends on the fact that $v''$ is a diagonal neighbor of $v'$. This means that $\rotv_{\rotv_w^4(v')}^{\primt} = \rotv_{v''}^{k\primt}$ as well, since $v''$ is a diagonal neighbor of $\rotv_w^4(v')$. We conclude $\rotv_{\rotv_w^4(v')}^{\primt} = \rotv_{v'}^{\primt}$ and hence $[\rotv_{v'}^{\primt},\rotv_w^{4}] = \id$. This implies $[\rotv_{v}^{\primt},\rotv_w^{4}] = \id$.
\end{proof}

Let $v \in \V$ and let $v'$ be a diagonal neighbor of $v$. We have $\rotv_{v}^{\primt} = \rotv_{v'}^{k\primt}$ for some $k$, uniquely defined modulo $\frac{q}{\primt}$. By Claim $2$ of Lemma~\ref{regularity}, we find $k$ is independent of choice of $v$ and $v'$ (as long as they are diagonal neighbors). In particular, we find $\rotv_{v'}^{\primt} = \rotv_v^{k\primt}$ and hence $k^2 \equiv 1 \mod \frac{q}{\primt}$. We will show something stronger, namely the following.

\begin{lem}\label{k=1}
Let $v \in \V$ and let $v'$ be a diagonal neighbor of $v$. Then $\rotv_v^{\primt} = \rotv_{v'}^{\primt}$.

Moreover, for any face $vuw$ of $\gr(\R)$, we have $\rotv_v^{\primt}\rotv_u^{\primt}\rotv_w^{\primt} = \id$.
\end{lem}

\begin{proof} Let $w$ be a neighbor of both $v$ and $v'$ such that $\rotv_w^2(v) = v'$. Label the neighbors of $w$ counterclockwise with the elements of $\ZZ/q\ZZ$ with $v_0 = v$.

We have $\rotv_{v_i}^{\primt} = \rotv_{v_{i+4m}}^{\primt}$ for any $m \in \ZZ$ using the fact that $\rotv_{v_i}^{\primt}$ commutes with $\rotv_w^4$ by the previous lemma.

Note that $\V_v \cap \dD{1}{w} = \{v_{2m}: m \in \ZZ\}$ and $\V_{v_1} \cap \dD{1}{w} = \{v_{2m+1}: m \in \ZZ\}$. We claim that $\rotv_v^{\primt}$ fixes $\V_{v_1} \cap \dD{1}{w}$ as a set. Note that it fixes $\V_{v_1}$ as a set since $\primt$ is even.

Suppose $\rotv_{v_0}^{\primt}(v_i)$ is not a neighbor of $w$ for some $i$. Then also $\rotv_{v_{4m}}^{\primt}(v_{i+4m})$ is not a neighbor of $w$ (because $\rotv_{v_{4m}}^{\primt}(v_{i+4m})$ is simply $\rotv_w^{4m}(\rotv_{v_0}^{\primt}(v_i))$, and rotations around $w$ preserve distance to $w$). Since $\rotv_{v_{4m}}^{\primt} = \rotv_v^{\primt}$, we find that $\rotv_{v_0}^{\primt}(v_{i+4m})$ is not a neighbor of $w$ for any $m$. This means there are at least $\frac{q}{4}$ distinct elements in $\V_{v_i}$ at distance $2$ of $w$. Moreover, there are $\frac{q}{2}$ distinct elements in $\V_{v_i}$ at distance $1$ of $w$. Hence we find $|\V_{v_i}| \geq \frac{q}{2} + \frac{q}{4} = \frac{3}{4}q > \frac{3}{8}|\V| > \frac{1}{3} |\V| = |\V_{v_i}|$, which gives a contradiction. This shows that $\rotv_v^{\primt}$ fixes $\V_{v_i} \cap \dD{1}{w}$ as a set for all $i$.

Suppose $\rotv_{v}^{\primt}(v_1) = v_{1+\primt_1}$ and $\rotv_v^{\primt}(v_{-1}) = v_{-1+\primt_{-1}}$. Note that $\rotv_v^{\primt}(v_{\pm1 + 4m}) = v_{\pm1+4m+\primt_{\pm1}}$, using $\rotv_v^{\primt} = \rotv_w^{4m}\rotv_v^{\primt}\rotv_w^{-4m}$.

Observe that $\rotv_w^{-\primt_1}\rotv_{v}^{\primt}$ fixes $v_1$, hence it is equal to $\rotv_{v_1}^x$ for some $x$. Consider the reflection $\sigma$ that maps $\overrightarrow{(v,w)}$ to $\overrightarrow{(w,v)}$. It fixes $v_1$. Conjugating the equality $\rotv_w^{-\primt_1}\rotv_{v}^{\primt} = \rotv_{v_1}^x$ with $\sigma$ gives $\rotv_v^{\primt_1}\rotv_{w}^{-\primt} = \rotv_{v_1}^{-x}$. Together, these give the equality $\rotv_w^{-\primt_1}\rotv_{v}^{\primt}\rotv_v^{\primt_1}\rotv_{w}^{-\primt} = \id$, or in other words, $\rotv_v^{\primt+\primt_1} = \rotv_w^{\primt_1+\primt}$. This means that the rotation $\rotv_v^{\primt+\primt_1}$ fixes $w$, which is a neighbor of $v$, and hence it is $\id$. We conclude $\primt_1 = -\primt$ modulo $q$. Analogously, we can show $\primt_{-1} = -\primt$ modulo $q$, so we find $\rotv_v^{\primt}(v_{2m+1}) = v_{2m+1-\primt}$ for all $m$.

Repeating this argument for $\rotv_{v'}^{\primt}$, we find $\rotv_{v'}^{\primt}(v_{2m+3}) = v_{2m+3-\primt}$ for all $m$. This means $\rotv_v^{\primt} = \rotv_{v'}^{\primt}$, as both of these fix $v$ and map $v_1$ to $v_{1-\primt}$. This shows the first part of the lemma.

For the second part, let $u = v_1$ and note that the equality $\rotv_w^{-\primt_1}\rotv_{v}^{\primt} = \rotv_{v_1}^x$ is simply the equality $\rotv_w^{\primt}\rotv_{v}^{\primt} = \rotv_{u}^x$. Conjugating with the rotation around $vuw$ that maps $v$ to $u$ yields $\rotv_v^{\primt}\rotv_u^{\primt} = \rotv_w^x$. Together, these equations give $\rotv_w^{\primt}\rotv_{w}^x\rotv_u^{-\primt} = \rotv_u^x$, and hence $\rotv_w^{\primt+x} = \rotv_u^{\primt+x}$. We conclude $x = -\primt$ modulo $q$ as before, and hence $\rotv_w^{\primt}\rotv_{v}^{\primt} = \rotv_u^{-\primt}$. From here, we easily find $\rotv_v^{\primt}\rotv_u^{\primt}\rotv_w^{\primt} = \id$, as was to be shown.
\end{proof}


\begin{lem}\label{j=q->q=2} Suppose $\primt = q$. Then $\R = \fer(1)$.
\end{lem}

\begin{proof} We have $\primi < q$ because $\R$ is not $\tet$. Hence we have $\primi = \frac{q}{2}$ and $\primi$ is odd. For all $v,v' \in \V$, we have $\rotv_{v'}^2(v) \in \V_v$, and hence $\rotv_{v'}^2\rotv_{v}^{\primi}\rotv_{v'}^{-2}$ is a rotation of order $2$ that fixes $v$, and hence it is $\rotv_{v}^{\primi}$. This means that $[\rotv_v^{\primi},\rotv_{v'}^{2}] = \id$ for all $v,v' \in \V$. Let $v,w \in \V$ be neighbors and consider $H = \langle \rotv_v^{\primi},\rotv_w^{\primi} \rangle$. Let $u = \rotv_v^{\primi}(w)$; note that $u\not\in \V_v\cup\V_w$. Now $\rotv_v^{\primi}$ maps $\V_w$ to $\V_u$ and vice versa. Likewise, $\rotv_w^{\primi}$ maps $\V_v$ to $\V_u$ and vice versa.


It is easily verified that $\rotv_v^{\primi}\rotv_w^{\primi}\rotv_v^{\primi} = \rotv_u^{\primi} = \rotv_w^{\primi}\rotv_v^{\primi}\rotv_w^{\primi}$ (where the latter part follows from the fact that $\rotv_u^{\primi} = \rotv_{u'}^{\primi}$ for all $u' \in \V_u$). We now easily verify that $H \cong \mathrm{Sym}(3)$ and $H$ is a normal subgroup of $\aut[\R]$ (any conjugation of $\rotv_v^{\primi}$, $\rotv_w^{\primi}$ is a rotation of order $2$ around some vertex, and all such rotations are elements of $H$).

Observe that $|Hv| = 3$; this follows from the fact that $|Hv|$ contains elements from $\V_v$, $\V_w$ and $\V_u$ and $\rotv_v^{\primi}$ fixes $v$.

The argument now splits into two cases: either $Hv = \{v,w,u\}$ or $w,u\not\in Hv$.  Assume the latter holds. Because $|Hv| = 3$ and $S_v^J(w)=u$, the orbit of the edge $vw$ has size $6$ and hence $H$ reverses no edges; we may apply Proposition~\ref{prop:modding}. This tells us that $\aut[\R]/H$ is the automorphism group of a reflexive regular map $\overline{\R}$ with $|\V(\overline{\R})| = \frac{|\V|}{3}$, $q(\overline{\R}) = \frac{q}{2}$ with simple graph. However, this means $\den(\overline{\R}) = \frac{3}{2}\den(\R) > \frac{3}{4}$, which is not possible. We conclude that $Hv = \{v,w,u\}$.

Let $u' \in \V_u$ such that $vwu'$ is a face of $\R$ with $\rotv_{u'}(v) = w$. Observe that $\rotv_u^{\primi}$ interchanges $v$ and $w$ and hence the same is true for $\rotv_{u'}^{\primi}$. We find $\rotv_{u'}(\overrightarrow{(u',v)}) = \overrightarrow{(u',w)} = \rotv_{u'}^{\primi}(\overrightarrow{(u',v)})$. It follows $\primi = 1$ and hence $q = 2$. Together with $p=3$, we immediately see $\R = \fer(1)$.

\end{proof}

With this lemma, the following proposition is now within our reach.

\begin{prop}\label{dens2/3} Suppose $\frac{1}{2} < \den(\R) \leq \frac{2}{3}$. Then $\den(\R) = \frac{2}{3}$.
\end{prop}

\begin{proof} We apply induction to the cardinality of $\V$. If $\R$ satisfies $|\V(\R)| \leq 3$, then $\R$ has density $\frac{2}{3}$ simply because $\frac{2}{3}$ is the only possible fraction $\frac{n}{d}$ satisfying $\frac{1}{2} < \frac{n}{d} \leq \frac{2}{3}$ and $d \leq 3$. Suppose $|\V| > 3$ and assume the proposition is true for all $\R'$ with $|\V(\R')| < |\V|$.

Consider $H = \langle \rotv_v^{\primt},\rotv_w^{\primt}\rangle$. Note that $H$ is a subgroup of $\aut[\R]$ of order $(\frac{q}{\primt})^2$, using Lemma~\ref{lem:primtand4commute} and the fact that $\rotv_{v}^{x\primt}\rotv_w^{y\primt} = \id$ if and only if both $x$ and $y$ are $0$ mod $\frac{q}{j}$. Note that any conjugation of $\rotv_v^{\primt}$ or $\rotv_w^{\primt}$ is a rotation around some element of $\V$ of order $\frac{q}{\primt}$. In other words, any conjugation of these elements is of the form $\rotv_u^{k\primt}$. If $u \in \V_v$ or $u\in \V_w$, this is an element of $H$ by definition of $\primt$. If $u\not\in \V_v\cap\V_w$, we have $\rotv_u^{k\primt} = \rotv_v^{-k\primt}\rotv_w^{-k\primt} \in H$ by Lemma~\ref{k=1}. We conclude that $H$ is normal in $\aut[\R]$. Moreover, $H$ is generated by any pair $\rotv_{v'}^{\primt},\rotv_{w'}^{\primt}$ with $\V_{v'},\V_{w'}$ distinct.

By Proposition~\ref{prop:modding}, $\aut[\R]/H$ corresponds to another regular map. Note that $|Hv| = \frac{q}{\primt}$ since $\rotv_v^{\primt}$ fixes $v$ and $|\langle \rotv_w^{\primt} \rangle v| = \frac{q}{\primt}$. By the previous remark, $|Hv'| = \frac{q}{\primt}$ for all $v' \in \V$.

Moreover, if two edges $e,e'$ are incident to $v$, then $He = He'$ if and only if $e' = \rotv_v^{k\primt}e$ for some $k \in \ZZ$.


Let $\overline{\R}$ be the regular map corresponding to $\aut[\R]/H$. By the above remarks, we have $|\V(\overline{\R})| = |\V|\frac{\primt}{q}$ and $q(\overline{\R}) = \primt$. Observe that $\overline{\R}$ has a simple graph. Indeed, suppose that two edges $He$, $He'$ in $\gr(\overline{\R})$ have the same vertices, say $Hv_1$ and $Hv_2$. Then in the original graph, we have $e = (v_1,v_2)$ and $e' = (v_1',v_2')$ with $Hv_1 = Hv_1'$ and $Hv_2 = Hv_2'$. Since $Hv_1 = Hv_1'$, we may assume $v_1' = v_1$ by picking another representative if necessary. Since there is an edge from $v_1$ to $v_2$, we have $\V_{v_1} \neq \V_{v_2}$, and hence we have $Hv_2 = \langle \rotv_{v_1}^{\primt}\rangle v_2$. This means that $v_2' = \rotv_{v_1}^{k\primt}(v_2)$ for some $k \in \ZZ$. Since $\gr(\R)$ is simple, we conclude $e' = \rotv_{v_1}^{k\primt}(e)$ as well, and hence $He = He'$.

We now find $\den(\overline{\R}) = \den(\R)$. Moreover, because we assumed $|\V| > 3$, we have $\R \neq \fer(1)$ and hence $\primt \neq q$ by Lemma~\ref{j=q->q=2}. This means $|\V(\overline{\R})| < |\V|$. By our induction hypotheses, we have $\den(\R) = \den(\overline{\R}) = \frac{2}{3}$ as desired. This concludes the proof.
\end{proof}

\begin{cor} If $\frac{1}{2} < \den(\R) \leq \frac{2}{3}$ and $v,w,u \in \V$ with $\V = \V_v\cup \V_w\cup\V_u$, then $\dD{1}{v} = \V_w\cup\V_u$.
\end{cor}

\begin{proof} We have $\dD{1}{v}\subseteq \V_w\cup\V_u$ by Lemma~\ref{lem:diagJ} (since $\R$ is not $\tet$, hence $\primi < q$), and $|\V_w\cup\V_u| = \frac{2}{3}|\V|$. As $|\dD{1}{v}| = q = \frac{2}{3}|\V|$ by the previous lemma, the result must hold.
\end{proof}

\begin{prop}\label{j=2} Suppose $\den(\R) = \frac{2}{3}$. Then $\primt = 2$.
\end{prop}

\begin{proof} Let $v,w \in \V$ be neighbors. Label the neighbors of $w$ by $0,1,\ldots,q-1$ clockwise with $v = 0$. By our assumption on $\den(\R)$, the elements of $\V_v$ are precisely elements of the form $2i$ mod $q$.

Let $u = \rotv_v(w)$. Suppose $\rotv_v(2i) = 2i'$. Then $\rotv_v(2i+\primt) = \rotv_v(\rotv_w^{\primt}(2i)) = \rotv_u^{\primt}(\rotv_v(2i)) = \rotv_u^{\primt}(2i') = 2i'-\primt$, using the fact that $\rotv_u^{\primt}\rotv_w^{\primt} = \rotv_v^{-\primt}$ acts as the identity on $\V_v$. This means $\rotv_v$ acts on congruence classes of $2i$ mod $\primt$. Moreover, we have $\rotv_v(2+k\primt) = -(2+k\primt)$ for any $k \in \ZZ$ because $\rotv_v(2) = \rotv_v(\rotv_w^2(v)) = \rotv_w^{-2}(v) = -2$.

Let $\iota \in \ZZ_{>0}$ be minimal such that $\rotv_v^{\iota}(2) \cong 2 \mod \primt$. Since the number of even congruence classes modulo $j$ is $\frac{\primt}{2}$, we have $\iota \leq \frac{\primt}{2}$.  We now have $\rotv_v^{\iota+1}(2) = -\rotv_v^{\iota}(2)$. Let $\sigma$ be the reflection that fixes $\overrightarrow{(v,w)}$. It maps $i$ to $-i$ and moreover, it maps $\rotv_v^{\iota}(2)$ to $\rotv_v^{-\iota+1}(2)$. In particular, we can conclude $\rotv_v^{\iota+1}(2) = -\rotv_v^{\iota}(2) = \sigma \rotv_v^{\iota}(2) =  \rotv_v^{-\iota+1}(2)$. It follows $\rotv_v^{2\iota}$ fixes $2$ and hence it fixes $\V_v$ pointwise. This means $\primi | 2\iota$ and hence $\primt | 2\iota$. So $\iota \geq \frac{\primt}{2}$. This however means $\iota = \frac{\primt}{2}$. In particular, $2,\rotv_v(2),\ldots,\rotv_v^{\iota-1}(2)$ are $\frac{\primt}{2}$ distinct even congruence classes modulo $\primt$, so it must be all of them.

Note however that $\rotv_v$ fixes the congruence class $0$ mod $\primt$. Since some power of $\rotv_v$ maps $2$ to an element that is $0$ mod $\primt$, we see that $0$ must the only even congruence class mod $\primt$. This is only possible if $\primt = 2$, as was to be shown.
\end{proof}

We now restate and prove our main theorem.

\begin{thm} Suppose $\R$ is a regular map with simple graph satisfying $\den(\R) > \frac{1}{2}$. Then either $\R = \tet$ or $\R = \fer(n)$ for some $n \in \ZZ_{>0}$.
\end{thm}

\begin{proof} Suppose $\R$ is not $\tet$.  We have $p=3$ by Proposition~\ref{Prop:p=3}. By Lemma~\ref{lem:diagJ} and Proposition~\ref{classes}, we find $q$ is even and $\den(\R) \leq \frac{2}{3}$. Write $q = 2n$. By Lemma~\ref{dens2/3}, we have $\den(\R) = \frac{2}{3}$ and by Proposition~\ref{j=2}, we find $\primt = 2$.

Let $v \in \V$ and let $w$ be a neighbor of $v$. Let $f$ be the face on the left of $\overrightarrow{(v,w)}$. We claim that $[\rotv_v,\rotf_f]^3 = \id$.

Note that $[\rotv_v,\rotf_f] = \rotv_v\rotv_w^{-1}$. We have $\rotv_v\rotv_w^{-1}(v) = \rotv_v^2(w)$. Note that $\rotv_w^{-1}(\rotv_v^2(w)) = \rotv_w(\rotv_v^2(w)) = \rotv_v^{-2}(w)$ using the fact that $\rotv_v^2(w) \in \V_w$. So $(\rotv_v\rotv_w^{-1})^2(v) = \rotv_v^{-1}(w)$. We have $\rotv_v\rotv_w^{-1}\rotv_v^{-1}(w) = v$, showing $[\rotv_v,\rotf_f]^3 = (\rotv_v\rotv_w^{-1})^3$ fixes $v$.

On the other hand, we have $[\rotv_v,\rotf_f]^{-3} = (\rotv_w\rotv_v^{-1})^3$. By symmetry of the situation, it fixes $w$. But then $[\rotv_v,\rotf_f]^3$ also fixes $w$. This immediately implies $[\rotv_v,\rotf_f]^3 = \id$ as claimed.

It now follows that $\autp[\R]$ is a quotient group of $\langle \rotv,\rotf\mid \rotv^{2n}= \id,\rotf^3 = \id,(\rotv \rotf)^2 = \id,[\rotv,\rotf]^3 = \id\rangle$ using $\rotv = \rotv_v$ and $\rotf = \rotf_f$. So $\autp[\R]$ is a quotient group of $\autp[\fer(n)]$. On the other hand, we have $|\autp[\fer(n)]| = 6n^2 = 3n\cdot2n = |\V|\cdot q = |\autp[\R]|$. This means $\autp[\R] = \autp[\fer(n)]$ and hence $\R = \fer(n)$, as was to be shown.
\end{proof}